\documentclass[12pt,a4paper]{article}
\usepackage{fullpage}
\usepackage[T2A]{fontenc}
\usepackage[utf8]{inputenc}
\usepackage[english,russian]{babel}
\usepackage{amsmath,amssymb,amsthm}

\theoremstyle{plain}
\newtheorem{theorem}{Theorem}
\newtheorem{lemma}{Lemma}
\newtheorem{corollary}{Corollary}

\theoremstyle{definition}
\newtheorem{remark}{Remark}
\newtheorem{definition}{Definition}

\begin{document}
\selectlanguage{english}

\title{Privileged coordinates for Carnot--Carath\'eodory
spaces of low smoothness}
\author{S.\,G. Basalaev\thanks{
The publication is supported by the Ministry of
Education and Science of the Russian Federation
(Project number 1.3087.2017/{\selectlanguage{russian}ПЧ}) and by the
International Mathematical Center of
Novosibirsk State University.
}}
\date{}

\maketitle

\begin{abstract}
We describe classes of coordinate systems in
Carnot--Carath\'{e}odory spaces of low smoothness
which allow for homogeneous approximations of quasimetrics
and basis vector fields. We establish the minimal smoothness
required for these classes to coinside with the class of 
the privileged coordinates described earlier for the
smooth case. We also apply these results to prove partial
analogues of existing results in the canonical
coordinates of the 2nd kind. As a geometric tool we prove
some convergence theorems in quasimetric spaces.
\end{abstract}

\section{Introduction}

Consider $C^\infty$-smooth connected Riemannian
$N$-dimensional manifold $\mathbb M$ with fixed distribution
$H \subset T \mathbb{M}$ and a scalar product
$\langle \cdot , \cdot \rangle: H \times H \to \mathbb R$
on it.
Recall that a commutator (or Lie bracket) of two vector
fields $X$, $Y$ is a vector field $[X,Y] = XY - YX$.
Commuting vector fields in $H$ iteratively we obtain
a family of foliations $H_1 = H$, $H_{k+1} = H + [H_k, H]$.
It is well known \cite{Rashevsky, Chow} that if the 
distribution $H$ is totally nonholonomic (i.\,e. there is
such $m > 0$ that $H_m = T \mathbb{M}$) then any two points 
of $\mathbb{M}$ can be connected by a \emph{horizontal curve} 
i.\,e. by an absolutely continuous curve $\gamma$ such that
$\dot\gamma \in H$ a.\,e.
The metric $d_{cc}$ on $\mathbb{M}$ defined as the infimum
of the lengths of horizontal curves is called the 
\emph{Carnot--Carath\'{e}odory metric} and the corresponding
metric spaces are \emph{Carnot Carath\'{e}odory spaces} or 
\emph{sub-Riemannian spaces}
(note, that the precise definition of these terms may differ
depending on the source).

Important subclass of such spaces are
\emph{equiregular} Carnot--Carath\'{e}odory spaces,
i.\,e. such spaces that in a filtration
\begin{equation}
\label{Eq:Filtration}
  H = H_1 \subsetneq H_2 \subsetneq \dots \subsetneq H_m
  = T \mathbb{M}
\end{equation}
the dimensions of $H_k(x)$ do not depend on $x$
(that is $H_k$ are distributions on $\mathbb{M}$).
In \cite{RoSt76} it is shown that every
Carnot--Carath\'{e}odory space can be locally lifted
to the equiregular space of a higher dimension.

Since we focus on the local properties of the equiregular
spaces it is convenient to choose a basis in $T\mathbb{M}$ 
subordinate to the structure~\eqref{Eq:Filtration}.
This means that in a neighborhood of $p \in \mathbb{M}$
we take vector fields $X_1, \dots, X_N$ such that
\[
  H_k(x) = \mathrm{span} \{ X_1(x), \dots,
    X_{\dim H_k}(x) \}.
\]
To each vector field $X_j$ we assign a formal weight
$\sigma_j = \min \{ k : X_j \in H_k \}$.

The crucial tool in studying the local geometry of
Carnot--Carath\'{e}odory spaces are nilpotent approximations.
The methods of nilpotent approximation originate in the
works on hypoelliptic operators and are formulated using
\emph{canonical coordinates of the 1st kind}
\begin{equation}
\label{Eq:1KindCoord}
  \theta_p(x_1, \dots, x_N) =
  \exp ( x_1 X_1 + \ldots + x_N X_N )(p).
\end{equation}

We formulate the key statements in the following theorem
(the wordings may differ from the ones given by authors,
see comparison of different formulations after the theorem):

\begin{theorem}[on the nilpotent approximation]
\label{Prop:NilpApprox}
Let $\mathbb{M}$ be an equiregular Carnot--Carath\'{e}odory
space, $p \in \mathbb{M}$.
Using coordinates~\eqref{Eq:1KindCoord} define a family
of anisotropic dilatations
\[
  \Delta_\varepsilon : \theta_p(x_1, \dots, x_N)
  \mapsto
  \theta_p(\varepsilon^{\sigma_1} x_1, \dots,
  \varepsilon^{\sigma_N} x_N).
\]

\noindent
\emph{1. (Rothschild--Stein theorem on local 
approximation \cite{RoSt76, Metivier}).}
There are uniform in a neighborhood of~$p$ limits
\[
  \widehat X_k(x) = \lim_{\varepsilon \to 0} \,
  (\Delta_\varepsilon^{-1})_* \,
  \varepsilon^{\sigma_k} X_k( \Delta_\varepsilon x ),
\]
and the homogeneous vector fields
$\widehat X_1, \ldots, \widehat X_N$ form a basis of the
Lie algebra of some Carnot group $\mathbb{G}^p$
(nilpotent graded stratified Lie group).

\noindent
\emph{2. (Ball--Box theorem by Nagel--Stein--Wainger 
\cite{NSW}).}
In a neighborhood $U$ of $p$ one can define a distance
function
\[
  d_\infty(x, y) =
  \max_{i = 1, \dots, N} |u_i|^{\frac{1}{\sigma_i}},
  \quad \text{ if }
  y = \theta_x(u_1, \dots, u_N).
\]
Then there are constants $0 < C_1 \leq C_2 < \infty$ 
such that
\[
  C_1 \, d_\infty(x, y) \leq d_{cc}(x, y)
  \leq C_2 \, d_\infty(x, y)
\]
for all $x, y \in U$.

\noindent
\emph{3. (Gromov's local approximation theorem \cite{Gromov}).}
In a neighborhood of $p$ there is a uniform limit
\[
  \widehat d_{cc}(x, y) = \lim_{\varepsilon \to 0}
  \frac{1}{\varepsilon} d_{cc}
  (\Delta_\varepsilon x, \Delta_\varepsilon y)
\]
where $\widehat d_{cc}$ is the Carnot--Carath\'{e}odory 
metric of the group $\mathbb{G}^p$ (i.\,e. formed by the
nilpotentized vector fields
$\widehat X_1, \ldots, \widehat X_N$).
\end{theorem}

In the original work by L.\,P.~Rothschild and
E.\,M.~Stein~\cite{RoSt76} it is shown that in the special
case of \emph{free} vector fields there is a decomposition
\[
  X_k(x) = \widehat X_k(x) + R_k(x),
\]
where vector fields $\widehat X_k$ are homogeneous and
the remainders $R_k$ are small as we approach~$p$.
The convergence of the vector fields to the homogeneous ones
in a smooth equiregular setting in the form stated above
is proved by G.~Metivier~\cite{Metivier}.
In a recent paper~\cite{Bramanti} the convergence is shown
using special \emph{regularized 1st kind coordinates}
in the case when $H \in C^{m-1, \alpha}$, $\alpha > 0$
(see also the generalization to even lower smoothness below).

Comparison of Carnot--Carath\'eodory metric with the distance
$d_\infty$ and a number of other distance functions
is conducted in~\cite{NSW}. It is worth noting that the
quantity $d_\infty$ is not a metric in general but
only a local quasimetric, i.\,e.
\[
  d_\infty(x,z) \leq Q(d_\infty(x, y) + d_\infty(y, z) ),
\]
for all $x, y, z \in U$ and some $Q = Q(U) \geq 1$.
In some works Ball--Box theorem is formulated as
\[
  \mathrm{Box}(x, C_1 r)
  \subset B_{cc}(x, r)
  \subset \mathrm{Box} (x, C_2 r)
\]
where $B_{cc}$ is a ball in the metric $d_{cc}$
and $\mathrm{Box}$ is a ball in quasimetric $d_\infty$.

The local approximation theorem is stated by 
M.~Gromov~\cite{Gromov} for ``smooth enough vector fields''
in the form
\[
  |d_{cc}(x, y) - \widehat d_{cc}(x, y)| = o(\varepsilon)
  \quad \text{ as } x, y \in B_{cc}(p, \varepsilon)
  \text{ and } \varepsilon \to 0.
\]
This statement is equivalent to the one stated above since
the metric $d_{cc}$ is homogeneous
w.r.t. $\Delta_\varepsilon$.

In~\cite{Vod07, KarmVod} the notion of equiregular
Carnot--Carath\'{e}odory spaces was generalized to
$C^1$-smooth setting. Following these works we use the
following definition:

\begin{definition}
Connected $C^\infty$-smooth manifold $\mathbb{M}$
of topological dimension~$N$ is called
\emph{$C^{k,\alpha}$-smooth equiregular
Carnot--Carath\'{e}odory space}, $k \in \mathbb{N}$,
$\alpha \in [0,1]$ (we designate $C^{k,0} = C^k$)
if there is a fixed filtration of the tangent bundle
$T \mathbb{M}$ by $C^{k,\alpha}$-smooth distributions
\begin{equation}
\label{Eq:NonSmoothFiltration}
  H_1 \subsetneq H_2 \subsetneq \ldots
  \subsetneq H_m = T \mathbb{M}
\end{equation}
such that $[H_i, H_j] \subset H_{i+j}$
for all $i, j = 1, \dots, m$.

Carnot--Carath\'{e}odory space
$\mathbb{M}$ is called a \emph{Carnot manifold}
under a stronger assumption
$H_{k} = \mathrm{span} \{ H_{k-1}, [H_i, H_j] : i+j=k \}$,
$k = 2, \dots, m$.

We refer to the number $m$ as the \emph{depth}
of Carnot--Carath\'{e}odory space.
\end{definition}

Note, that a classical sub-Riemannian space with
$H \in C^{k+m-1,\alpha}$ is a $C^{k,\alpha}$-smooth
Carnot manifold. Carnot manifolds carry
Carnot--Carath\'{e}odory metric
(it is proved in \cite{KarmVod} for the smoothness class
$C^{1,\alpha}$, $\alpha \in (0, 1]$, 
and in \cite{BasVod} for $C^1$).
But generally it is not required for two points in a
Carnot--Carath\'{e}odory space with given definition
to be connectable by a horizontal curve. In this case the
local quasimetric $d_\infty$ is used to describe their 
geometric behavior. The properties of $C^1$-smooth
Carnot--Carath\'{e}odory spaces are stated in the following
theorem (compare with Theorem~\ref{Prop:NilpApprox}).

\begin{theorem}[\cite{KarmVod, Greshnov, KarmQuasi}]
\label{Th:KarmVod}
Let $C^1$-smooth vector fields $X_1, \ldots, X_N$
conform to the commutator table
\[
  [X_i, X_j](x) =
  \sum_{k : \: \sigma_k \leq \sigma_i + \sigma_j}
  c_{ijk}(x) X_k(x).
\]
Then

1. There is a family of vector fields
$\widehat X'_1, \ldots, \widehat X'_N$ in $\mathbb{R}^N$ 
such that \linebreak
$\exp (u_1 \widehat X'_1 + \ldots + u_N \widehat X'_N)(0)
= (u_1, \ldots, u_N)$ and
\[
  [\widehat X'_i, \widehat X'_j](u) =
  \sum_{k : \: \sigma_k = \sigma_i + \sigma_j}
  c_{ijk}(p) \widehat X'_k(u).
\]
The vector fields
$\widehat X'_1, \ldots, \widehat X'_N$
form a structure of the graded nilpotent Lie algebra
(and the structure the Carnot algebra for Carnot manifolds).

2. For vector fields
$\widehat X_k = (\theta_p)_* \widehat X'_k$ we have
\[
  \widehat X_k(x) = \lim_{\varepsilon \to 0}
  (\Delta^p_\varepsilon)^{-1}_* \varepsilon^{\sigma_k}
  X_k(\Delta^p_\varepsilon x)
\]
where the limit is uniform in a neighborhood of $p$.

3. Using the vector fields $\widehat X'_k$ construct a
quasimetric $\widehat d'_\infty$ the same way as $d_\infty$
is built and push-forward it to the manifold:
$\widehat d_\infty(x, y) =
\widehat d'_\infty(\theta_p^{-1}(x), \theta_p^{-1}(y))$
($\widehat d_\infty$ can not be defined in the same 
straightforward way using $\widehat X_k$ since these fields
are only continuous in general). Then
\[
  |d_\infty(x, y) - \widehat d_\infty(x,y)| = o(\varepsilon)
\]
as $x, y \in \mathrm{Box}(p, \varepsilon)$ and
$\varepsilon \to 0$ uniformly in a neighborhood of $p$.
\end{theorem}

The convergence of the vector fields to the homogeneous ones
in the coordinates of the 1st kind is established for
$C^{1,\alpha}$-smooth equiregular spaces in~\cite{KarmVod}
and for $C^1$-smooth spaces in~\cite{Greshnov}.
The convergence of the quasimetric $d_\infty$ to the 
$C^1$-smooth equiregular spaces is shown in~\cite{KarmQuasi}. 
Note that for $C^1$-smooth spaces the coordinate system 
$\theta_p$ is also just $C^1$-smooth.

All aforementioned results rely on canonical coordinates
of the 1st kind. There are problems, however, for which
some other coordinate systems may be more suitable.
For instance, the canonical coordinates of the 2nd kind
\[
  (x_1, \ldots, x_N) \mapsto
  \exp (x_N X_N) \circ \ldots \circ
  \exp (x_1 X_1) (p)
\]
are heavily used in~\cite{BasVod}.

This leads us to the question: what conditions should we
impose on the coordinate system so that 
Theorem~\ref{Prop:NilpApprox} still holds?
In a paper~\cite{Bell} by A. Bella\"{\i}che such condition
is stated for smooth sub-Riemannian spaces.

\begin{theorem}
\label{Prop:BallBox}
The nilpotent approximation in the sense of
Theorem~\ref{Prop:NilpApprox} exists in coordinates~$\phi_p$
around the point $p$ if and only if
\[
  \phi_p \bigl( \mathrm{Box}(0, C_1 \varepsilon) \bigr)
  \subset
  B_{d_{cc}}(p, \varepsilon)
  \subset
  \phi_p \bigl( \mathrm{Box}(0, C_2 \varepsilon) \bigr)
\]
for some $0 < C_1 \leq C_2 < \infty$ and for all
$0 < \varepsilon \leq \varepsilon_0$. Here
$\mathrm{Box}(0, r) = \{ x \in \mathbb{R}^N :
|x_k|^{\sigma_k} \leq r \}$.
\end{theorem}

In this work we provide, as a byproduct, an independent proof 
of this result in a smooth equiregular case
(see Section~\ref{Sec:Smooth}), however, as
Remark~\ref{Rem:Sin1} shows, this assertion may fail
in low-smoothness setting.
We describe classes of coordinate systems
$\phi_p$ for which partial analogs of 
Theorem~\ref{Prop:NilpApprox} hold in
Carnot--Carath\'{e}odory spaces of low smoothness.

Since convergence theorems are already proved in canonical
coordinates of the 1st kind~$\theta_p$,
we focus on the properties of the \emph{transition map}
$\Phi_p = \phi_p^{-1} \circ \theta_p$, i.\,e. we obtain
conditions on the transition map to the new coordinates
which preserves limits of Theorem~\ref{Prop:NilpApprox}.
In smooth case, as one can conclude from 
Theorem~\ref{Prop:BallBox}, it is necessary and sufficient
that the coordinate change map $\Phi$ enjoys
\begin{equation}
\label{Eq:IntroCond1}
  \mathrm{Box}(0, C_1 \varepsilon)
  \subset \Phi(\mathrm{Box}(0, \varepsilon))
  \subset \mathrm{Box}(0, C_2 \varepsilon),
\end{equation}
for some $0 < C_1 \leq C_2 < \infty$ and for all
$\varepsilon \in (0, \varepsilon_0)$. In 
Section~\ref{Sec:Homogeneous} we show that when the transition
map does not have enough smoothness the 
condition~\eqref{Eq:IntroCond1} is still necessary
(Theorem~\ref{Th:HomoNecessary}) but is not enough 
(Remark~\ref{Rem:Sin1}).
Then in Section~\ref{Sec:Homogeneous} we obtain sufficient
condition on coordinate change that preserves the homogeneous
limit of quasimetrics (Theorem~\ref{Th:Condition1}),
namely:
if 1) $\Phi$ is a homeomorphism, 2) there exists
a limit
\begin{equation}
\label{Eq:IntroCond2}
  L(x) := \lim_{\varepsilon \to 0}
  \delta_\varepsilon^{-1} \circ \Phi
  \circ \delta_\varepsilon (x),
\end{equation}
uniform in neighborhood of the origin,
where $\delta_\varepsilon(x_1, \dots, x_N) = 
(\varepsilon^{\sigma_1} x_1, \dots,
\varepsilon^{\sigma_N} x_N)$,
and 3) $L$ is a homeomorphism as well then in new coordinates 
there is a uniform limit of quasimetrics (analogous
to the limit of metrics in Theorem~\ref{Prop:NilpApprox})
and $L$ is the isometry between limiting quasimetrics
in the original and new coordinate systems.

In Section~\ref{Sec:Nilpotent} we show that
condition~\eqref{Eq:IntroCond2} is still insufficient
for obtaining homogeneous limits of the basis vector fields 
(Remark~\ref{Rem:Sin2}).
Next, we obtain such sufficient condition
(Theorem~\ref{Th:Condition2}), namely:
if 1) $\Phi$ is a $C^1$-diffeomorphism,
2) there is a uniform limit
\begin{equation}
\label{Eq:IntroCond3}
  \lambda(x) := \lim_{\varepsilon \to 0}
  D \delta_\varepsilon^{-1} \circ D \Phi
  \circ D \delta_\varepsilon(x),
\end{equation}
and 3) $\det \lambda(0) \ne 0$,
then in new coordinate system there is a homogeneous
limit of vector fields
(same as in Theorem~\ref{Prop:NilpApprox})
and $\lambda$ is an isomorphism of homogeneous Lie algebras
between the original and new coordinate systems.

In Section~\ref{Sec:Smooth} we prove that in case when
$\Phi \in C^m$ where $m$ is the space depth all conditions
\eqref{Eq:IntroCond1}, \eqref{Eq:IntroCond2} and \eqref{Eq:IntroCond3} are equivalent. Examples from
Remarks~\ref{Rem:Sin1} and~\ref{Rem:Sin2} show that
when smoothness is lower all three conditions are distinct.

In Section~\ref{Sec:2Kind} it is proved that some
canonical coordinate systems (including 2nd kind coordinates) 
enjoy condition~\eqref{Eq:IntroCond2} in $C^1$-smooth case
and condition~\eqref{Eq:IntroCond3} in $C^m$-smooth.

\section{Homogeneous approximation of quasimetric spaces}
\label{Sec:Homogeneous}

\begin{definition}
Let $U \subset \mathbb{R}^N$ be a neighborhood of the origin.
By a \emph{quasimetric} on $U$ we mean such function 
$d : U \times U \to \mathbb{R}$ that

\begin{itemize}
\item
$d$ is continuous;

\item
$d(x, y) \geq 0$ for all $x, y \in U$;
$d(x, y) = 0$ $\Leftrightarrow$ $x = y$;

\item
$d(x, y) \leq C d(y, x)$ for all $x, y \in U$
and some $C \geq 1$;

\item
$d(x, z) \leq Q (d(x,y) + d(y,z))$ for all $x, y, z \in U$
and some $Q \geq 1$.
\end{itemize}

The pair $(U, d)$ is called a \emph{quasimetric space}.
\end{definition}

\begin{definition}
Let $(\sigma_1, \dots, \sigma_N)$ be a tuple of positive real
numbers. Introduce in $\mathbb{R}^N$ an one-parameter 
\emph{dilatation group}
\[
  \delta_\varepsilon(x_1, \dots, x_N)
  = (\varepsilon^{\sigma_1} x_1, \dots,
  \varepsilon^{\sigma_N} x_N),
  \quad \varepsilon > 0.
\]
\end{definition}

\begin{definition}
Define on $\mathbb{R}^N$
\emph{$\delta_\varepsilon$-homogeneous quasinorm}
\[
  \Vert x \Vert
  = \Vert (x_1, \dots, x_N) \Vert
  = \max\limits_{k = 1, \dots, N} |x_k|^{\frac{1}{\sigma_k}}.
\]
As $\mathrm{Box}(r)$ we denote a set of $x \in \mathbb{R}^N$
such that $\Vert x \Vert < r$.
Note, that
$\delta_\varepsilon \mathrm{Box}(r) = \mathrm{Box}(\varepsilon r)$.
\end{definition}

\begin{definition}
A quasimetric $\widehat d$ on $\mathbb{R}^N$ is called
$\delta_\varepsilon$-homogeneous if
\[
  \widehat d(\delta_\varepsilon x, \delta_\varepsilon y)
  = \varepsilon \, d(x, y)
\]
for all $x, y \in \mathbb{R}^N$, $\varepsilon > 0$.
A triple
$(\mathbb{R}^N, \delta_\varepsilon, \widehat d)$
is called \emph{$\delta_\varepsilon$-homogeneous
quasimetric space}.
\end{definition}

\begin{definition}
We say that the quasimetric space
$(\mathbb{R}^N, \delta_\varepsilon, \widehat d)$ is the
\emph{$\delta_\varepsilon$-homogeneous approximation} for
$(U, d)$ if there is a limit
\begin{equation}
\label{Eq:QuasiLimit}
  \lim\limits_{\varepsilon \to 0} \frac{1}{\varepsilon}
  d(\delta_\varepsilon x, \delta_\varepsilon y)
  = \widehat d(x, y)
\end{equation}
uniform in a neighborhood of the origin. If this limit exists
we say that $(U, d)$ has the $\delta_\varepsilon$-homogeneous 
approximation.
\end{definition}

\begin{remark}
In terminology of the work~\cite{Seliv} the space
$(\mathbb{R}^N, \widehat{d})$ is a \emph{local tangent
cone} to the quasimetric space $(U, d)$.
This notion generalizes the notion of a tangent cone to
a metric space introduced by M.\,Gromov~\cite{GromovMetric}.
We do not use that terminology since the functional approach
of these works defines the tangent cone up to isometry.
Our approach, albeit more na\"{\i}ve, distinguishes
between homogeneous approximations in different coordinates 
even if the resulting spaces are isometric.
\end{remark}

\begin{lemma}
\label{Lemma:DistBounds}
If quasimetric $d$ has the $\delta_\varepsilon$-homogeneous 
approximation then there are constants
$0 < C_1 \leq C_2 < \infty$ and $r_0 > 0$ such that
\[
  C_1 \| x \|  \leq  d(0, x)  \leq  C_2 \| x \|
\]
for all $x \in \mathrm{Box}(r_0)$.
\end{lemma}

\begin{proof}
Let $\delta_\varepsilon$-homogeneous
quasimetric $\widehat d$ be the
$\delta_\varepsilon$-homogeneous approximation for $d$.
Then there is $r_0 > 0$ such that
\begin{equation}
\label{Eq:ZeroConvergence}
  \frac{1}{\varepsilon} d(0, \delta_\varepsilon x)
  \to \widehat d(0, x)
  \quad \text{as } \varepsilon \to 0
\end{equation}
uniformly in $x \in \mathrm{Box}(2 r_0)$.

Let
$m = \inf \{ \widehat d(0, v) : \Vert v \Vert = r_0 \}$,
$M = \sup \{ \widehat d(0, v) : \Vert v \Vert = r_0 \}$.
There is $\varepsilon_1 > 0$ such that for all
$v \in \partial \mathrm{Box}(r_0)$ and
$\varepsilon < \varepsilon_1$ we have
\[
  \frac{m}{2}
  < \frac{1}{\varepsilon} d(0, \delta_\varepsilon v)
  < 2 M.
\]
Consequently, for all
$\varepsilon < \min \{ \varepsilon_1, 2 \}$ and
$x \in \partial \mathrm{Box}(\varepsilon r_0)$
it holds
\[
  \frac{m}{2 r_0} \Vert x \Vert
  = \frac{m}{2} \varepsilon
  < d(0, x)
  < 2 M \varepsilon
  = \frac{2 M}{r_0} \Vert x \Vert.
\]
Thus, the lemma is proved.
\end{proof}

\begin{theorem}[necessary condition of the homogeneous
approximation in new coordinates]
\label{Th:HomoNecessary}
Let $(U,d)$ be a quasimetric space that has
$\delta_\varepsilon$-homogeneous approximation,
$\Phi : U \to \Phi(U)$ be a homeomorphism and $\Phi(0)=0$. Define quasimetric
$\rho$ on $\Phi(U)$ as
\[
  \rho(u, v) = d(\Phi^{-1}(u), \Phi^{-1}(v)).
\]

If quasimetric space $(\Phi(U), \rho)$ has
$\delta_\varepsilon$-homogeneous approximation then there are
\linebreak
$\varepsilon_0 > 0$ and $0 < C_1 \leq C_2 < \infty$ such that
\begin{equation}
\label{Eq:HomoNecessary}
  \mathrm{Box}(C_1 \varepsilon)
  \subset
  \Phi( \mathrm{Box}(\varepsilon) )
  \subset
  \mathrm{Box}(C_2 \varepsilon)
\end{equation}
for all $\varepsilon \in (0, \varepsilon_0)$.
\end{theorem}

\begin{proof}
By Lemma~\ref{Lemma:DistBounds}
for quasimetrics $d$ and $\rho$ there are positive constants
$r_1$, $r_2$, $c_1$, $c_2$, $c_3$, $c_4$ such that
\[
  c_1 \| x \| \leq d(0, x) \leq c_2 \| x \|,
  \quad
  c_3 \| y \| \leq \rho(0, y) \leq c_4 \| y \|.
\]
for all $x \in \mathrm{Box}(r_1)$, $y \in \mathrm{Box}(r_2)$.
Since $\Phi$ is a homeomorphism of a neighborhood of the
origin, there are positive constants $r_3 \leq r_2$ and
$r_4 \leq r_1$ such that
$\mathrm{Box}(r_3) \subset \Phi^{-1}(\mathrm{Box}(r_2))$
and $\mathrm{Box}(r_4) \subset \Phi(\mathrm{Box}(r_3))$.
Then for all $x \in \mathrm{Box}(r_4)$ we have
\[
  \frac{c_1}{c_4} \| x \|
  \leq \frac{1}{c_4} d(0, x)
  = \frac{1}{c_4} \rho(0, \Phi(x))
  \leq \| \Phi(x) \|
  \leq \frac{1}{c_3} \rho(0, \Phi(x))
  = \frac{1}{c_3} d(0, x)
  \leq \frac{c_2}{c_3} \| x \|.
\]
The theorem follows.
\end{proof}

\begin{remark}
\label{Rem:Sin1}
Let us demonstrate that conditions of 
Theorem~\ref{Th:HomoNecessary} are not sufficient in the
general case. Consider a plane $\mathbb{R}^2$ with coordinates 
$(x,y)$, dilatation
$\delta_\varepsilon(x,y) = (\varepsilon x, \varepsilon^2 y)$
and the $\delta_\varepsilon$-homogeneous metric
\[
  d \bigl( (x_1, y_1), (x_2, y_2) \bigr)
  = \sqrt{(x_1 - x_2)^2 + |y_1 - y_2|}.
\]
Consider a coordinate change $\Phi(x, y) = (x, y + f(x))$
where
\[
  f(x) =
  \begin{cases}
  \frac{x^2}{2} \sin \frac{1}{|x|^{1-\beta}}, & x \ne 0, \\
  0, & x = 0,
  \end{cases}
\]
for some $\beta \in (0, 1)$.
Then $f \in C^{1,\beta} \setminus C^2$ and $f'(0) = 0$.
Note, that $D\Phi(0) = \mathrm{Id}$ and as a consequence
$\Phi$ is a $C^{1,\beta}$-diffeomorphism of a neighborhood
of the origin. Also, the following estimate holds:
\[
  \frac{1}{2} (|x|^2 + |y|)
  \leq \frac{|x|^2}{2} + |y|
  \leq |\Phi_1(x,y)|^2 + |\Phi_2(x,y)|
  \leq \frac{3}{2} |x|^2 + |y|
  \leq \frac{3}{2}( |x|^2 + |y| ).
\]
But it is easy to see that the metric
$\rho(u, v) = d(\Phi(u), \Phi(v))$ does not have a
$\delta_\varepsilon$-homogeneous approximation.
Indeed,
\[
  \frac{1}{\varepsilon} d \bigl(
    \Phi(\varepsilon x_1, \varepsilon^2 y_1),
    \Phi(\varepsilon x_2, \varepsilon^2 y_2)
  \bigr) = \sqrt{
    (x_1 - x_2)^2
    + | y_1 - y_2 - \tfrac{1}{\varepsilon^2}
    (f(\varepsilon x_1) - f(\varepsilon x_2)) |
  },
\]
where the expression
\[
  \frac{1}{\varepsilon^2}
  (f(\varepsilon x_1) - f(\varepsilon x_2))
  = \frac{x_1}{2} \sin \frac{1}{|\varepsilon x_1|^{1-\beta}}
  - \frac{x_2}{2} \sin \frac{1}{|\varepsilon x_2|^{1-\beta}}
\]
has no limit when $x_1 \ne x_2$ as $\varepsilon \to 0$.

It is possible to provide analogous example for functions
of class $C^{1,1}$. For instance, one can consider
$f(x) = \int_0^x t \sin \frac{1}{t} \, dt$.

For $C^2$-smooth maps on this particular metric space,
however, the conditions of Theorem~\ref{Th:HomoNecessary}
are sufficient. We'll show this in 
Lemma~\ref{Lemma:SmoothCond}.
\end{remark}

\begin{lemma}
\label{Lemma:Condition1}
Let $U \subset \mathbb{R}^N$ be a neighborhood of the origin,
$\Phi : U \to \mathbb{R}^N$ be a continuous mapping.
There is an uniform in a neighborhood of the origin limit
\begin{equation}
\label{Eq:NonSmoothCond1}
  L(x) := \lim_{\varepsilon \to 0}
  \delta_\varepsilon^{-1} \circ \Phi \circ
  \delta_\varepsilon (x)
\end{equation}
if and only if there is a continuous
$\delta_\varepsilon$-homogeneous mapping
$L : \mathbb{R}^N \to \mathbb{R}^N$ such that
\begin{equation}
\label{Eq:NonSmoothCond1L}
  \Phi_k(x) = L_k(x) + o(\varepsilon^{\sigma_k}),
  \quad
  k = 1, \ldots, N,
\end{equation}
as $\varepsilon \to 0$ and
$x \in \mathrm{Box}(\varepsilon)$.

If either of these conditions is fulfilled and both mappings
$\Phi$ and $L$ are homeomorphisms then there is also a limit
\[
  \lim_{\varepsilon \to 0}
  \delta_\varepsilon^{-1} \circ \Phi^{-1} \circ \delta_\varepsilon(y)
  = L^{-1}(y)
\]
and it is uniform in a neighborhood of the origin.
\end{lemma}

\begin{proof}
Let the uniform limit~\eqref{Eq:NonSmoothCond1} exist for
$x \in \mathrm{Box}(r_0)$. Then the limiting map $L$
is continuous and for all $t \in (0, 1]$ the following holds
\begin{equation}
\label{Eq:LMustBeHomo}
  L_k(\delta_t x)
  = \lim_{\varepsilon \to 0}
  \frac{1}{\varepsilon^{\sigma_k}}
  \Phi_k(\delta_\varepsilon \delta_t x)
  = \lim_{\varepsilon \to 0}
  t^{\sigma_k} \frac{1}{(t\varepsilon)^{\sigma_k}}
  \Phi_k(\delta_{t \varepsilon} x)
  = t^{\sigma_k} L_k(x),
\end{equation}
i.\,e. $\delta_t \circ L = L \circ \delta_t$.
We can extend $L$ by homogeneity to a continuous
$\delta_\varepsilon$-homogeneous mapping on $\mathbb{R}^N$.
Moreover, for all
$x \in \mathrm{Box}(r_0)$ we have
\[
  \Phi_k (\delta_\varepsilon x)
  = \varepsilon^{\sigma_k} \bigl(
    \varepsilon^{-\sigma_k} \Phi_k (\delta_\varepsilon(x)) 
    \bigr)
  = \varepsilon^{\sigma_k} \bigl( L(x) + o(1) \bigr)
  = L(\delta_\varepsilon x) + o(\varepsilon^{\sigma_k}).
\]

Conversely, let the condition~\eqref{Eq:NonSmoothCond1L} hold.
Fix $r_0 > 0$ such that
$\overline{\mathrm{Box}}(r_0) \subset U$.
Then
\begin{equation}
\label{Eq:LimPhiEIsUniform}
  \frac{1}{\varepsilon^{\sigma_k}}
  \Phi_k(\delta_\varepsilon x)
  = \frac{1}{\varepsilon^{\sigma_k}} \bigr(
    L_k(\delta_\varepsilon x) + o(\varepsilon^{\sigma_k}) 
  \bigr)
  = L_k(x) + o(1)
\end{equation}
as $\varepsilon \to 0$ uniformly in
$x \in \mathrm{Box}(r_0)$.
Thus, conditions \eqref{Eq:NonSmoothCond1} and
\eqref{Eq:NonSmoothCond1L} are equivalent.

Next, let $\Phi$ and $L$ be homeomorphisms.
Since the map $L$ is continuous and
$\delta_\varepsilon$-homogeneous, we have
\[
  M = \sup_{x \ne 0}
  \frac{\Vert L(x) \Vert}{\Vert x \Vert}
  = \sup_{x \ne 0}
  \Bigl\Vert \delta^{-1}_{\Vert x \Vert} L(x) \Bigr\Vert
  = \sup_{x \ne 0}
  \Bigl\Vert L \Bigl(
    \delta^{-1}_{\Vert x \Vert} x
  \Bigr) \Bigr\Vert
  = \sup_{\Vert v \Vert = 1}
  \Vert L(v) \Vert < \infty.
\]
Since $L(x) \ne 0$ when $x \ne 0$, by the same reasoning
we have
\[
  m = \sup_{x \ne 0}
  \frac{\Vert x \Vert}{\Vert L(x) \Vert}
  = \sup_{\Vert v \Vert = 1}
  \frac{1}{\Vert L(v) \Vert} < \infty.
\]
Therefore, two-sided estimates
\[
  \frac{1}{m} \Vert x \Vert
  \leq \Vert L(x) \Vert \leq M \Vert x \Vert,
  \quad
  \frac{1}{M} \Vert y \Vert
  \leq \Vert L^{-1}(y) \Vert \leq m \Vert y \Vert,
  \quad
\]
hold for all $x, y \in \mathbb{R}^N$. Consequently, there is a
neighborhood of the origin $V$ such that two-sided
estimates
\[
  \frac{1}{2m} \Vert x \Vert
  \leq \Vert \Phi(x) \Vert \leq 2M \Vert x \Vert,
  \quad
  \frac{1}{2M} \Vert y \Vert
  \leq \Vert \Phi^{-1}(y) \Vert \leq 2m \Vert y \Vert,
  \quad
\]
hold for all $x \in V$, $y \in \Phi(V)$.
Let $\Phi_\varepsilon(x) = \delta_\varepsilon^{-1} \circ
\Phi \circ \delta_\varepsilon(x)$ and let
$r_1 \leq r_0$ be such that $\mathrm{Box}(r_1) \subset V$,
$\mathrm{Box}(\frac{r_1}{2m}) \subset \Phi(V)$.
Then
\[
  \Vert \Phi^{-1}_\varepsilon(y) \Vert
  = \frac{1}{\varepsilon} \Vert \Phi^{-1}
    (\delta_\varepsilon y) \Vert
  \leq \frac{2m}{\varepsilon} \Vert \delta_\varepsilon y \Vert
  = 2m \Vert y \Vert
\]
for all $y \in \mathrm{Box}(\frac{r_1}{2m})$
and $\varepsilon > 0$, i.\,e.
$\Phi^{-1}_\varepsilon( \mathrm{Box}(\frac{r_1}{2m}) )
\subset \mathrm{Box}(r_1)$.
From~\eqref{Eq:LimPhiEIsUniform}
it follows that
$\Vert \Phi_\varepsilon(x) - L(x) \Vert = o(1)$
as $\varepsilon \to 0$
uniformly in $x \in \mathrm{Box}(r_0)$. Therefore
\[
  y - L \circ \Phi^{-1}_\varepsilon (y)
  = \Phi_\varepsilon(\Phi_\varepsilon^{-1}(y))
    - L (\Phi_\varepsilon^{-1} (y))
  = o(1)
\]
as $\varepsilon \to 0$ uniformly in
$y \in \mathrm{Box}(\frac{r_1}{2m})$.
Since the continuous map $L^{-1}$ is uniformly continuous
on $\overline{\mathrm{Box}}(\frac{r_1}{2m})$,
we have
\[
  L^{-1}(y) - \Phi^{-1}_\varepsilon(y)
  = L^{-1}(y) - L^{-1} ( L \circ \Phi^{-1}_\varepsilon(y))
  = o(1)
\]
as $\varepsilon \to 0$ uniformly on
$\mathrm{Box}(\frac{r_1}{2m})$.
The lemma is proved.
\end{proof}

\begin{theorem}[sufficient condition of the homogeneous
approximation in new coordinates]
\label{Th:Condition1}
Let $U \subset \mathbb{R}^N$ be a neighborhood of the origin,
$d$ be a continuous quasimetric on $U$ and quasimetric
$\widehat d$ be its $\delta_\varepsilon$-homogeneous
approximation.
Let $\Phi : U \to \mathbb{R}^N$ be a homeomorphism
on a neighborhood of the origin
such that there is a $\delta_\varepsilon$-homogeneous
homeomorphism $L : \mathbb{R}^N \to \mathbb{R}^N$
enjoying the condition
\[
  L(x) = \lim_{\varepsilon \to 0}
  \delta_\varepsilon^{-1} \circ \Phi \circ 
  \delta_\varepsilon(x),
\]
as $\varepsilon \to 0$ uniformly in
$x \in \mathrm{Box}(r_0)$.
Consider quasimetric space $(\Phi^{-1}(U), \rho)$
where $\rho(u, v) = d(\Phi(u), \Phi(v))$.
Then

$1.$
There is a limit
\[
  \widehat \rho (u,v) := \lim_{\varepsilon \to 0}
  \frac{1}{\varepsilon}
  \rho(\delta_\varepsilon u, \delta_\varepsilon v)
\]
for all $u,v \in \mathbb{R}^N$ uniform in a neighborhood
of the origin $V \subset \Phi^{-1}(U)$.

$2.$ The mapping $\widehat \rho(u,v)$ is a continuous
$\delta_\varepsilon$-homogeneous quasimetric on~$\mathbb{R}^N$.

$3.$ The map $L$ is a $\delta_\varepsilon$-homogeneous
isometry of quasimetric spaces\footnote{isoquasimetry?}
$(\mathbb{R}^N, \widehat d)$ and
$(\mathbb{R}^N, \widehat \rho)$, i.\,e.
\[
  \delta_\varepsilon L(x) = L(\delta_\varepsilon x),
  \quad
  \widehat \rho(x,y) = \widehat d(L(x), L(y))
\]
for all $x, y \in \mathbb{R}^N$.
\end{theorem}

\begin{proof}
Indeed, let
$\Phi_\varepsilon = \delta^{-1}_\varepsilon \circ
\Phi \circ \delta_\varepsilon$.
Then
\[
  \rho(\delta_\varepsilon u, \delta_\varepsilon v)
  = d \bigl(
    \Phi(\delta_\varepsilon u), \Phi(\delta_\varepsilon v)
  \bigr)
  = d \bigl(
    \delta_\varepsilon \circ \Phi_\varepsilon (u),
    \delta_\varepsilon \circ \Phi_\varepsilon (v)
  \bigr).
\]
From Lemma~\ref{Lemma:Condition1} derive
\begin{multline*}
  \frac{1}{\varepsilon}
  \rho(\delta_\varepsilon u, \delta_\varepsilon v)
  = \frac{1}{\varepsilon} d \bigl(
    \delta_\varepsilon \circ \Phi_\varepsilon (u),
    \delta_\varepsilon \circ \Phi_\varepsilon (v)
  \bigr)
  = \widehat{d} \bigl(
    \Phi_\varepsilon(u), \Phi_\varepsilon(v)
  \bigr) + o(1) \\
  = \widehat{d} \bigl(
    L(u) + o(1), L(v) + o(1) \bigr) + o(1),
\end{multline*}
where all $o(1)$ are uniform in $u,v$ in a neighborhood
of the origin. Thus,
\[
  \widehat{\rho}(u,v)
  := \lim_{\varepsilon \to 0} \frac{1}{\varepsilon}
  \rho(\delta_\varepsilon u, \delta_\varepsilon v)
  = \widehat{d} \bigl( L(u), L(v) \bigr).
\]
Since $L$ is a homeomorphism, $\widehat{\rho}$
is also a quasimetric on $\mathbb{R}^N$. Besides,
\[
  \widehat{\rho}(\delta_t u, \delta_t v)
  = \widehat{d} \bigl( L(\delta_t u), L(\delta_t v) \bigr)
  = \widehat{d} \bigl( \delta_t \circ L(u),
     \delta_t \circ L(v) \bigr)
  = t \, \widehat{d} \bigl( L(u), L(v) \bigr)
  = t \, \widehat{\rho}(u, v)
\]
for all $t > 0$ and $u, v \in \mathbb{R}^N$.
This ends the proof.
\end{proof}

\begin{remark}
Note, that conditions of Theorem~\ref{Th:Condition1} are
not necessary. It may happen that the required limit
is non-uniform but still generates an isomorphism.
Moreover, we can provide an example of the map $\Phi$ such
that the limit~\eqref{Eq:NonSmoothCond1} does not exist
but quasimetrics still converge in a new coordinate system.
Consider $\mathbb{C}$ with the Euclidean metric
$d(z,w) = |z - w|$ and homothetic dilatation
$\delta_\varepsilon(z) = \varepsilon z$, $\varepsilon > 0$.

Let the map $\Phi : \mathbb{C} \to \mathbb{C}$
be defined as
\[
  \Phi(r e^{i\theta}) = r \, e^{i(\theta + \ln r)},
  \quad
  \Phi(0) = 0,
\]
where $\theta \in [0, 2 \pi]$. Then $\Phi$ is continuous
since $\Phi(r e^{0i}) = \Phi(r e^{2 \pi i})$ and
$\Phi(r e^{i\theta}) \to 0$ as $r \to 0$.
The metric $d(z, w) = |\Phi(z) - \Phi(w)|$
is homogeneous:
\begin{multline*}
  \frac{1}{\varepsilon} \bigl|
    \Phi(\varepsilon r_1 \, e^{i \theta_1}) -
    \Phi(\varepsilon r_2 \, e^{i \theta_2})
  \bigr|
  = \frac{1}{\varepsilon} \bigl|
    \varepsilon r_1 \,
      e^{i \theta_1 + i \ln(\varepsilon r_1)}
    - \varepsilon r_2 \,
      e^{i\theta_2 + i \ln(\varepsilon r_2)}
  \bigr| \\
  = |e^{i \ln \varepsilon}| \bigl|
    r_1 \, e^{i\theta_1 + i \ln r_1}
    -r_2 \, e^{i\theta_2 + i \ln r_2}
  \bigr|
  = \bigl| \Phi(r_1 \, e^{i \theta_1})
  - \Phi(r_2 \, e^{i \theta_2}) \bigr|.
\end{multline*}
However, the expression~\eqref{Eq:NonSmoothCond1} for this
map becomes
\[
  \frac{1}{\varepsilon} \Phi(\varepsilon z)
  = \frac{1}{\varepsilon} \Phi(\varepsilon r \, e^{i \theta})
  = r \, e^{i \theta + i \ln (\varepsilon r)}
  = r \, e^{i \theta + i \ln r} e^{i \ln \varepsilon}
  = \Phi(z) e^{i \ln \varepsilon}
\]
and diverges as $\varepsilon \to 0$ and $z \ne 0$. Note, that
metrics $|z - w|$ and $|\Phi(z) - \Phi(w)|$ are not
isometric.
\end{remark}

Let us apply the results of this section to
Carnot--Carath\'{e}odory spaces.
Let $\mathbb{M}$ be $C^1$-smooth equiregular space.
In a neighborhood $U$ of $p \in \mathbb{M}$ choose a basis 
$X_1, \dots, X_N$ subordinate to the filtration 
\eqref{Eq:NonSmoothFiltration}.
Recall, that using the canonical 1st kind coordinates
\[
  \theta_x(u_1, \ldots, u_N) =
  \exp (u_1 X_1 + \ldots + u_N X_N)(x),
  \quad x \in U
\]
we define the quasimetric
$d_\infty(x, y) = \max\limits_{k = 1, \dots, N}
|u_k|^{\frac{1}{\sigma_k}}$
and the family of dilatations
\[
  \Delta^p_\varepsilon : \theta_p(u_1, \dots, u_N)
  \mapsto \theta_p (\varepsilon^{\sigma_1} u_1,
  \dots, \varepsilon^{\sigma_N} u_N).
\]

From Theorem~\ref{Th:KarmVod} and Theorem~\ref{Th:Condition1}
follows

\begin{corollary}
\label{Corollary:CC1}
Let $\mathbb{M}$ be $C^1$-smooth Carnot--Carath\'{e}odory
space, $X_1, \dots, X_N$ be a basis in a neighborhood of
$p \in \mathbb{M}$, subordinate to 
\eqref{Eq:NonSmoothFiltration}, $\theta_p$ be the canonical
1st kind coordinates~\eqref{Eq:1KindCoord} and
$\phi_p : U \subset \mathbb{R}^N \to \mathbb{M}$ be
homeomorphism of a neighborhood of the origin to a
neighborhood of $p$. Define a family of dilatations
\[
  \widetilde \Delta^p_\varepsilon :
  \phi_p(x_1, \dots, x_N) \mapsto
  \phi_p(\varepsilon^{\sigma_1} x_1, \dots,
  \varepsilon^{\sigma_N} x_N).
\]

If there is a uniform limit
\[
  L(x) := \lim_{\varepsilon \to 0}
  \delta_\varepsilon^{-1}
  \circ \phi_p^{-1} \circ \theta_p \circ 
  \delta_\varepsilon(x),
\]
and $L$ is a homeomorphism then there is a limit
\[
  \widetilde d_\infty(x, y) = \lim_{\varepsilon \to 0}
  \frac{1}{\varepsilon} d_\infty(
    \widetilde \Delta^p_\varepsilon x,
    \widetilde \Delta^p_\varepsilon y
  )
\]
uniform in a neighborhood of the origin and
$\widetilde d_\infty$ is a
$\widetilde \Delta^p_\varepsilon$-homogeneous quasimetric,
isometric to $\widehat d_\infty$. The isometry is given by the
map $\mathcal{L} = \phi_p \circ L \circ \theta_p^{-1}$:
$\widehat d_\infty(x, y) =
\widetilde d_\infty(\mathcal{L} x, \mathcal{L} y)$.
\end{corollary}

\begin{remark}
\label{Rem:Sin2}
If in the previous corollary $\mathcal{L}$ is a
$C^1$-diffeomorphism, we can define vector fields
$\widetilde X_j = \mathcal{L}_* \widehat X_j$.
These vector fields are homogeneous w.r.t.\ the
dilatation in new coordinates, but in general they can not
be obtained as homogeneous limits of vector fields
$\Phi_* X_j$ if $\Phi$ does not have enough regularity.
Consider, e.g., $\mathbb{R}^2_{x,y}$ with the family
of vector fields
$\{ \frac{\partial}{\partial_x}, \frac{\partial}{\partial y} \}$, the dilatation
$\delta_\varepsilon(x,y) = (\varepsilon x, \varepsilon^2 y)$
and the transition map
$\Phi(x,y) = (x, y + f(x))$ where
\[
  f(x) =
  \begin{cases}
    x^3 \sin \frac{1}{x}, & x \ne 0, \\
    0, & x = 0.
  \end{cases}
\]
Then $\Phi$ is a $C^{1,1}$-diffeomorphism of a neighborhood
of the origin and
\[
  \delta_\varepsilon^{-1} \circ \Phi \circ \delta_\varepsilon
  (x,y)
  = \begin{pmatrix}
    x \\
    y + \varepsilon x^3 \sin \frac{1}{\varepsilon x}
  \end{pmatrix}
  \to
  \begin{pmatrix} x \\ y \end{pmatrix}
\]
as $\varepsilon \to 0$. However,
$
  \Phi_* \frac{\partial}{\partial x}
  = \frac{\partial}{\partial x}
  + (3x^2 \sin \frac{1}{x} - x \cos \frac{1}{x})
    \frac{\partial}{\partial y}
$
and the expression
\[
  (\delta_\varepsilon)^{-1}_* \varepsilon \Phi_*
  \frac{\partial}{\partial x}(\varepsilon x, \varepsilon^2 y)
  = \frac{\partial}{\partial x} +
  \Big(
    3 \varepsilon x^2 \sin \frac{1}{\varepsilon x} -
    x \cos \frac{1}{\varepsilon x}
  \Big)
  \frac{\partial}{\partial y}
\]
diverges as $\varepsilon \to 0$ if $x \ne 0$.
\end{remark}

\section{Homogeneous approximation of vector fields}
\label{Sec:Nilpotent}

In this section we propose sufficient condition on the
transition map $\Phi$ which preserves the homogeneous
approximations of basis vector fields of a
Carnot--Carath\'{e}odory space in new coordinates.

\begin{definition}
Let the dilatation $\delta_\varepsilon$ be defined in
$U \subset \mathbb{R}^N$. We say that continuous vector field
$X$ on $U$ has $\delta_\varepsilon$-homogeneous approximation
of degree $r$ if there is a limit
\begin{equation}
\label{Eq:FieldLimit}
  \widehat X(x) := \lim_{\varepsilon \to 0}
  \, (\delta_\varepsilon^{-1})_* \,
  \varepsilon^r X(\delta_\varepsilon x)
\end{equation}
uniform in the neighborhood of the origin. Note, that
the vector field~$\widehat X$ is
$\delta_\varepsilon$-homogeneous of degree $r$.
\end{definition}

\begin{remark}
It is clear that if the limit~\eqref{Eq:FieldLimit}
exists for some $r$ then for all $r' > r$ this limit vanishes.
Therefore, it makes sense to assign a formal degree to
a vector field as the infimum of $r$ such that the 
limit~\eqref{Eq:FieldLimit} is zero. For basis vector fields
of the Carnot--Carath\'{e}odory space this definition of
a degree coinsides with the previous one.
\end{remark}

\begin{lemma}
\label{Lemma:Condition2}
Let $\Phi \in C^1(U, \mathbb{R}^N)$.
The limit
\[
  \lambda(x) := \lim_{\varepsilon \to 0}
  D\delta^{-1}_\varepsilon \circ D\Phi \circ
  D\delta_\varepsilon(x)
\]
uniform in a neighborhood of the origin exists if and only if
for all $k, l \in \{ 1, \dots, N \}$ such that
$\sigma_k > \sigma_l$ there are continuous functions
$\lambda_{kl} : \mathbb{R}^N \to \mathbb{R}$
$\delta_\varepsilon$-homogeneous of degrees
$\sigma_k - \sigma_l$ respectively and such that
\begin{equation}
\label{Eq:Condition2}
  \frac{\partial \Phi_k}{\partial x_l}(x)
  = \lambda_{kl}(x) + o(\varepsilon^{\sigma_k - \sigma_l})
\end{equation}
as $\varepsilon \to 0$ and all $o(\cdot)$ are uniform
$x \in \mathrm{Box}(\varepsilon)$.

Under conditions of this lemma if also $\Phi(0) = 0$
then there is a uniform limit
\[
   L(x) := \lim_{\varepsilon \to 0}
   \delta_\varepsilon^{-1} \circ \Phi
     \circ \delta_\varepsilon(x)
\]
and $\lambda = DL$.
\end{lemma}

\begin{proof}
Since $D\delta_\varepsilon$ is a diagonal matrix with
$\varepsilon^{\sigma_1}, \dots, \varepsilon^{\sigma_N}$
on diagonal, we have
\[
  [D\delta_\varepsilon^{-1} \circ D\Phi
  \circ D\delta_\varepsilon]_{kl} (x)
  = \varepsilon^{\sigma_l - \sigma_k}
  \frac{\partial \Phi_k}{\partial x_l}(\delta_\varepsilon x).
\]

Let $V = \mathrm{Box}(r_0)$ and for all
$k, l \in \{ 1, \dots, N \}$ let us have the uniform limits
\[
  \lambda_{kl}(x)
  := \lim_{\varepsilon \to 0}
  \varepsilon^{\sigma_l - \sigma_k}
  \frac{\partial \Phi_k}{\partial x_l}(\delta_\varepsilon x),
  \quad
  x \in V.
\]
Then the functions $\lambda_{kl}$ are continuous and
$\lambda_{kl}(\delta_t x) = t^{\sigma_k - \sigma_l} 
\lambda_{kl}(x)$
for $x \in V$, $t \in (0, 1]$. We can extend the functions
$\lambda_{kl}$ by homogeneity to the functions defined
on $\mathbb{R}^N$.

Conversely, let the condition \eqref{Eq:Condition2} hold.
Then in the case $\sigma_l \geq \sigma_k$ we have
\[
  \varepsilon^{\sigma_l - \sigma_k}
  \frac{\partial \Phi_k}{\partial x_l}(\delta_\varepsilon x)
  \to \begin{cases}
    \frac{\partial \Phi_k}{\partial x_l}(0),
    & \sigma_l = \sigma_k, \\
    0, & \sigma_l > \sigma_k,
  \end{cases}
\]
as $\varepsilon \to 0$ uniformly in
a compact neighborhood of the origin. When
$\sigma_l < \sigma_k$ we have
\[
  \varepsilon^{\sigma_l - \sigma_k}
  \frac{\partial \Phi_k}{\partial x_l}(\delta_\varepsilon x)
  = \varepsilon^{\sigma_l - \sigma_k} \bigl(
    \lambda_{kl}(\delta_\varepsilon x)
    + o(\varepsilon^{\sigma_k - \sigma_l})
  \bigr)
  = \lambda_{kl}(x) + o(1)
\]
where $o(\cdot)$ is uniform in $x$.

If $\Phi(0) = 0$ then
$\delta_\varepsilon^{-1} \circ \Phi \circ
\delta_\varepsilon(0) = 0 = L(0)$. By the classical analysis
results on the uniform convergence of derivatives
there is a map $L : U \to \mathbb{R}^N$ in a neighborhood
of the origin $U$ such that $\lambda = DL$ and
$\delta_\varepsilon^{-1} \circ \Phi \circ
\delta_\varepsilon(x) \to L(x)$ as $\varepsilon \to 0$
uniformly in $U$. The lemma is proved.
\end{proof}

\begin{theorem}[Sufficient condition of approximation
of vector fields in new coordinates]
\label{Th:Condition2}
Let $X$ be a continuous vector field in a neighborhood of
the origin $U$ and for some $r > 0$ let there be a
uniform limit
\[
  \widehat X(x) = \lim_{\varepsilon \to 0}
  (\delta_\varepsilon^{-1})_* \, \varepsilon^{r}
  X (\delta_\varepsilon x).
\]

Let $\Phi : U \to \mathbb{R}^N$ be a $C^1$-diffeomorphism
on a neighborhood of the origin such that $\Phi(0) = 0$,
there is a limit
\[
  \lambda(x) := \lim_{\varepsilon \to 0}
  D\delta^{-1}_\varepsilon \circ D\Phi \circ
  D\delta_\varepsilon(x)
\]
uniform in a neighborhood of the origin,
and $\det \lambda(0) \ne 0$.
Denote $Y(y) = \Phi_* X (\Phi^{-1}(y))$.
Then there is a uniform limit
\[
  \widehat Y(y) = \lim_{\varepsilon \to 0}
  (\delta_\varepsilon^{-1})_* \, \varepsilon^{r}
  Y (\delta_\varepsilon y)
\]
and $\widehat Y(y) = L_* \widehat X(L^{-1}(y))$
where the map $L$ is defined by~\eqref{Eq:NonSmoothCond1}.
\end{theorem}

\begin{proof}
By Lemma~\ref{Lemma:Condition2} in a neighborhood
of the origin there are uniform limits
\[
  L(x) = \lim_{\varepsilon \to 0}
  \delta_\varepsilon^{-1} \circ \Phi
    \circ \delta_\varepsilon(x),
  \quad
  L_*(x) = DL(x) = \lim_{\varepsilon \to 0}
  (\delta_\varepsilon^{-1} \circ \Phi
    \circ \delta_\varepsilon)_*(x).
\]
Since $\det DL(0) = \det \lambda(0) \ne 0$,
$L(x)$ is a diffeomorphism of neighborhoods of the origin.
By Lemma~\ref{Lemma:Condition1} in a neighborhood of the
origin there is a uniform limit
\[
  L^{-1}(y) = \lim_{\varepsilon \to 0}
  \delta_\varepsilon^{-1} \circ \Phi^{-1}
    \circ \delta_\varepsilon(y).
\]

Thus, in a neighborhood small enough
\begin{multline*}
  (\delta_\varepsilon^{-1})_* \, \varepsilon^{r}
  Y (\delta_\varepsilon y)
  = (\delta_\varepsilon^{-1})_* \, \varepsilon^{r}
  \Phi_* X (\Phi^{-1} (\delta_\varepsilon y)) \\
  = (\delta_\varepsilon^{-1} \circ \Phi
    \circ \delta_\varepsilon)_* (\delta_\varepsilon^{-1})_*
    \varepsilon^{r} X \circ \delta_\varepsilon
    (\delta_\varepsilon^{-1} \circ \Phi^{-1}
      \circ \delta_\varepsilon(y))
  \to L_* \widehat X (L^{-1}(y))
\end{multline*}
as $\varepsilon \to 0$ uniformly in $y$.
\end{proof}

From Theorem~\ref{Th:KarmVod} and Theorem~\ref{Th:Condition2}
immediately follows
\begin{corollary}
\label{Corollary:CC2}
Let $\mathbb{M}$ be $C^1$-smooth Carnot--Carath\'{e}odory
space, $X_1, \dots, X_N$ be a basis in a neighborhood of
$p \in \mathbb{M}$, subordinate 
to~\eqref{Eq:NonSmoothFiltration},
$\theta_p$ be canonical 1st kind 
coordinates~\eqref{Eq:1KindCoord} and
$\phi_p : U \subset \mathbb{R}^N \to \mathbb{M}$ be
$C^1$-diffeomorphism of a neighborhood of the origin on
a neighborhood of $p$. Define a family of dilatations
\[
  \widetilde \Delta^p_\varepsilon :
  \phi_p(x_1, \dots, x_N) \mapsto
  \phi_p(\varepsilon^{\sigma_1} x_1, \dots, 
  \varepsilon^{\sigma_N} x_N).
\]

If there is a limit
\[
  \lambda(x) := \lim_{\varepsilon \to 0}
  D\delta_\varepsilon^{-1}
  \circ D\phi_p^{-1} \circ D\theta_p \circ 
  D\delta_\varepsilon(x)
\]
uniform in a neighborhood of the origin and
$\det \lambda(0) \ne 0$ then

$1)$ In a neighborhood of the origin there are uniform limits
\[
  \widetilde X_k(x) = \lim_{\varepsilon \to 0}
  (\widetilde \Delta^p_\varepsilon)^{-1}_* \varepsilon^{d_k}
  X_k(\widetilde \Delta^p_\varepsilon x).
\]

$2)$ The conditions of Corollary~\ref{Corollary:CC1}
are fulfilled, the maps $L$ and $\mathcal{L}$ are
continuously differentiable and
$\widetilde X_k = \mathcal{L}_* \widehat X_k$.
\end{corollary}

\section{Transition map and smoothness}
\label{Sec:Smooth}

In this section we derive bounds on smoothness of a
transition map $\Phi$ under which the necessary
condition~\eqref{Eq:HomoNecessary} becomes sufficient.

In the following we use the usual multiindex notations.
If $\alpha = (\alpha_1, \dots, \alpha_N)$ where
$\alpha_k$ are nonnegative integers, $k = 1, \dots, N$,
then denote
\[
  |\alpha| = \alpha_1 + \dots + \alpha_N,
  \quad
  \alpha ! = \alpha_1! \cdot \ldots \cdot \alpha_N!,
  \quad
  x^\alpha = x_1^{\alpha_1} \cdot \ldots \cdot x_N^{\alpha_N},
  \quad
  D^\alpha \Phi = \frac{\partial \Phi^{|\alpha|}}
    {\partial x_1^{\alpha_1} \dots \partial x_N^{\alpha_N}}.
\]
Introduce also the \emph{weight} of multiindex
$\sigma(\alpha) =
\alpha_1^{\sigma_1} + \ldots + \alpha_N^{\sigma_N}$.

\begin{lemma}
\label{Lemma:SmoothCond}
Let the map $\Phi : U \to \mathbb{R}^N$ be such that
$\Phi_k \in C^{\sigma_k}(U)$ for $k = 1, \dots, N$.
The following conditions are equivalent:

\begin{enumerate}
\item
There are constants $C > 0$ and
$\varepsilon_0 > 0$ such that
$\Phi(\mathrm{Box}(\varepsilon)) \subset
\mathrm{Box}(C \varepsilon)$
for all $0 < \varepsilon \leq \varepsilon_0$.

\item
For coordinate functions of the map $\Phi$ it holds
$\Phi_k(x) = O(\varepsilon^{\sigma_k})$
as $\varepsilon \to 0$ and
$x \in \mathrm{Box}(\varepsilon)$.

\item
$D^\alpha \Phi_k(0) = 0$ for all multiindices
$\alpha$ such that $\sigma(\alpha) < \sigma_k$.

\item
There are
uniform limits
\begin{align}
  \label{Eq:PhiLimit}
  L(x) & = \lim_{\varepsilon \to 0}
  \delta_\varepsilon^{-1} \circ \Phi
  \circ \delta_\varepsilon (x),
  \\
  \label{Eq:DPhiLimit}
  DL(x) & = \lim_{\varepsilon \to 0}
  D\delta_\varepsilon^{-1} \circ D\Phi
  \circ D\delta_\varepsilon (x).
\end{align}
\end{enumerate}

Under these conditions the coordinate functions
of the map $L$ are polynomials.
\end{lemma}

\begin{proof}
Equivalence of conditions 1 and 2 is clear since condition 2
is just a rephrase of condition 1 in terms of coordinates.
To see that condition 3 is equivalent to condition 2
Taylor expand coordinate functions of the map $\Phi$
to corresponding orders:
\begin{equation}
\label{Eq:Taylor}
  \Phi_k(x)
  = P_k(x) + o(|x|^{\sigma_k})
  = \sum_{\alpha:\: |\alpha| \leq \sigma_k}
  \frac {D^\alpha \Phi_k(0)} {\alpha!} x^\alpha
  + o(|x|^{\sigma_k}).
\end{equation}
Note, that $x^\alpha = O(\varepsilon^{\sigma(\alpha)})$
if $x \in \mathrm{Box}(\varepsilon)$.
Then $\Phi_k(x) = O(\varepsilon^{\sigma_k})$ if and only if
$D^\alpha \Phi_k(0) = 0$ for all $\alpha$ such that
$\sigma(\alpha) < \sigma_k$.

Now prove that condition 3 is necessary and sufficient for
existence of the limit~\eqref{Eq:PhiLimit}.
Indeed, using expansion~\eqref{Eq:Taylor}
we derive
\[
  \frac{1}{\varepsilon^{\sigma_k}}
  \Phi_k(\delta_\varepsilon x) =
  \sum_{\alpha : \: \sigma(\alpha) < \sigma_k}
  \frac {D^\alpha \Phi_k(0)}
  {\alpha! \, \varepsilon^{\sigma_k - \sigma(\alpha)}} 
  x^\alpha
  + \sum_{\alpha : \: \sigma(\alpha) = \sigma_k}
  \frac {D^\alpha \Phi_k(0)}
  {\alpha!} x^\alpha
  + o(1).
\]
This expression converges as $\varepsilon \to 0$
if and only if the first term vanishes, i.\,e. condition~3
is fulfilled. The second term in this case is the expression
for the coordinate function $L_k(x)$ of the limiting map $L$.

Next, consider limit~\eqref{Eq:DPhiLimit}.
Since $D\delta_\varepsilon$ is a diagonal matrix with
$\varepsilon^{\sigma_1}, \dots, \varepsilon^{\sigma_N}$
on diagonal, we have
\[
  [ D\delta_\varepsilon^{-1} \circ D\Phi
  \circ D\delta_\varepsilon ]_{kl} (x)
  = \varepsilon^{\sigma_l - \sigma_k}
  \frac{\partial \Phi_k}{\partial x_l}(\delta_\varepsilon x).
\]
Since
$\frac{\partial \Phi_k}{\partial x_l} \in C^{\sigma_k - 1}(U)$,
we have
$\frac{\partial \Phi_k}{\partial x_l}(x)
= \frac{\partial P_k}{\partial x_l}(x)
+ o(|x|^{\sigma_k - 1})$,
where $P_k$ is the Taylor polynomial for $\Phi_k$ defined 
by~\eqref{Eq:Taylor}. Note, that
$\frac{\partial (x^\alpha)}{\partial x_l}$ is a
$\delta_\varepsilon$-homogeneous monomial of degree
$\sigma(\alpha) - \sigma_l$. Thus,
\[
  \varepsilon^{\sigma_l - \sigma_k}
  \frac{\partial \Phi_k}{\partial x_l}(\delta_\varepsilon x)
  = \sum_{\alpha : \: \sigma(\alpha) < \sigma_k}
  \frac {D^\alpha \Phi_k(0)}
    {\alpha! \, \varepsilon^{\sigma_k - \sigma(\alpha)}}
  \frac{\partial (x^\alpha)}{\partial x_l}
  + \sum_{\alpha : \: \sigma(\alpha) = \sigma_k}
  \frac {D^\alpha \Phi_k(0)}{\alpha!}
  \frac{\partial (x^\alpha)}{\partial x_l}
  + o(1).
\]
Again, the expression converges as $\varepsilon \to 0$
if and only if the first term vanishes, i\,e. condition~3
holds. The second term in this case is
$\frac{\partial L_k}{\partial x_l}(x)$.
The lemma is proved.
\end{proof}

\begin{corollary}
\label{Corollary:CC3}
Let $\mathbb{M}$ be $C^m$-smooth Carnot--Carath\'{e}odory
space of the depth~$m$, $p \in \mathbb{M}$,
$\theta_p$ be the coordinates of the 1st kind in the 
neighborhood of $p$.
$C^m$-smooth coordinate system $\phi_p$ in the neighborhood 
of $p$ enjoys the conditions of Corollaries~\ref{Corollary:CC1}
and~\ref{Corollary:CC2} if and only if there are constants
$0 < C_1 \leq C_2 < \infty$ and $\varepsilon_0 > 0$ such that
\[
  \phi_p(\mathrm{Box}(C_1 \varepsilon))
  \subset
  \theta_p(\mathrm{Box}(\varepsilon))
  \subset
  \phi_p(\mathrm{Box}(C_2 \varepsilon))
\]
for all $\varepsilon \in (0, \varepsilon_0)$.
\end{corollary}

\begin{proof}
Necessity immediately follows from 
Theorem~\ref{Th:HomoNecessary}.
Applying Lemma~\ref{Lemma:SmoothCond} we see that
$\Phi = \theta_p^{-1} \circ \phi_p$
enjoys conditions of Theorems~\ref{Th:Condition1} 
and~\ref{Th:Condition2}.
\end{proof}

\begin{remark}
From examples in Remarks~\ref{Rem:Sin1} and~\ref{Rem:Sin2}
it follows that this assertion on the smoothness of
the map~$\Phi$ can not be weakened in general case,
i.\,e. conditions of Theorems~\ref{Th:HomoNecessary},
\ref{Th:Condition1} and \ref{Th:Condition2} are the same
thing for $C^m$-smooth maps, and become distinct as soon
as the smoothness is a bit lower (like $C^{m-1,1}$).
\end{remark}

\begin{remark}
For $C^1$-smooth Carnot manifolds the Ball--Box Theorem holds 
(clause~2 of Theorem~\ref{Prop:NilpApprox},
see proof, e.\,g., in \cite[Theorem 8]{KarmVodIndam}).
This theorem together with Corollary~\ref{Corollary:CC3}
proves Theorem~\ref{Prop:BallBox} given in the introduction
in the case of $C^m$-smooth Carnot manifolds.
\end{remark}

\section{Canonical coordinate systems}
\label{Sec:2Kind}

Let $\mathbb{M}$ be $C^k$-smooth Carnot--Carath\'{e}odory
space,
$X_1, \dots, X_N$ be the basis of $T\mathbb{M}$ in
a neighborhood of $p \in \mathbb{M}$ subordinate to
the filtration \eqref{Eq:NonSmoothFiltration}.
Split the family of vector fields $\{ X_i \}_{i=1}^N$ into
$L$ disjoint subfamilies
$\{ X_{j,1}, \ldots, X_{j,k_j} \}$, $j=1, \ldots, L$,
and consider the mapping
\begin{multline}
\label{Eq:2KindCoord}
  \phi_p(u_1, \ldots, u_N)
  = \exp (u_{L,1} X_{L,1} + \ldots + u_{L,k_L} X_{L,k_L})
  \circ \ldots \\
  \ldots \circ
  \exp (u_{2,1} X_{2,1} + \ldots + u_{2,k_2} X_{2,k_2})
  \circ
  \exp (u_{1,1} X_{1,1} + \ldots + u_{1,k_1} X_{1,k_1})(p).
\end{multline}
Then $\phi_p \in C^k$ and
$\frac{\partial \phi_p}{\partial u_i}(0) = X_i(p)$.
Consequently, $\phi_p$ is a $C^k$-diffeomorphism of a 
neighborhood of the origin on a neighborhood of $p$.
The particular case of such mapping is the 
\emph{canonical coordinate system of the 2nd kind}:
\[
 \theta^2_p(u_1, \dots, u_N) =
 \exp (u_N X_N) \circ \exp (u_{N-1} X_{N-1}) \circ \ldots
 \circ \exp (u_1 X_1) (p).
\]

\begin{theorem}
\label{Th:2KindCoord}
Let $\mathbb{M}$ be a $C^1$-smooth Carnot--Carath\'{e}odory
space, $p \in \mathbb{M}$.
Define using the coordinate system~\eqref{Eq:2KindCoord} 
an one-parametric family of dilatations
\[
   \Delta^p_\varepsilon :
   \phi_p(x_1, \ldots, x_N) \mapsto 
   \phi_p(\varepsilon^{\sigma_1} x_1, \ldots,
     \varepsilon^{\sigma_N} x_N).
\]
Then there is a limit
\[
  \widetilde d_\infty(x, y) = \lim_{\varepsilon \to 0}
  \frac{1}{\varepsilon} d_\infty
  (\Delta^p_\varepsilon x, \Delta^p_\varepsilon y)
\]
uniform in a neighborhood of $p$
and $\widetilde d_\infty$ is a
$\Delta^p_\varepsilon$-homogeneous quasimetric
isometric to quasimetric $\widehat d_\infty$
from Theorem~\ref{Th:KarmVod}.
If $\mathbb{M}$ is a $C^m$-smooth space of the depth~$m$
then there are uniform in the neighborhood of $p$ limits 
of the scaled vector fields
\[
  \widetilde X_k(x) = \lim_{\varepsilon \to 0}
  \, (\Delta^p_\varepsilon)^{-1}_* \varepsilon^{d_k}
  X_k(\Delta^p_\varepsilon x),
\]
and the vector fields $\widetilde X_k$ define the structure
of the nilpotent graded Lie algebra isomorphic
to the algebra from Theorem~\ref{Th:KarmVod}.
\end{theorem}

To prove this theorem we use the following result.

\begin{theorem}[\cite{KarmQuasi}]
\label{Th:CurveDivergence}
Let $\mathbb{M}$ be $C^1$-smooth Carnot--Carath\'{e}odory
space, $p \in \mathbb{M}$, $X_1, \dots, X_N$ be a 
basis in a neighborhood of~$p$, subordinate to the
structure~\eqref{Eq:NonSmoothFiltration}, and
let $\widehat X_k$, $k = 1, \dots, N$, be the nilpotent
approximations of these vector fields, built using
the canonical 1st kind coordinates as in 
Theorem~\ref{Th:KarmVod}. Then there is a neighborhood
$U$ of $p$ such that for any two absolutely continuous
curves $\gamma, \widehat\gamma : [0, 1] \to \mathbb{M}$
such that $\gamma(0) = \widehat\gamma(0) \in U$ and
\[
  \dot\gamma(t) = \sum_{i=1}^N b_i(t) X_i(\gamma(t)),
  \quad
  \dot{\widehat\gamma}(t) = \sum_{i=1}^N
  b_i(t) \widehat X_i(\widehat\gamma(t)),
\]
where measurable functions $b_i(t)$ meet the property
\begin{equation}
\label{Eq:MeasClass}
  \int\limits_0^1 |b_i(t)| \, dt < S \varepsilon^{\sigma_i},
  \quad S < \infty, \quad i = 1, \dots, N,
\end{equation}
we have
\[
  \max \{ d_\infty(\gamma(1), \widehat\gamma(1)),
  \widehat d_\infty(\gamma(1), \widehat\gamma(1)) \}
  \leq o(1) \cdot \varepsilon,
\]
where $o(1)$ is uniform in $U$ and in all collections of
functions $\{ b_i(t) \}_{i=1}^N$ with the
property~\eqref{Eq:MeasClass}.
\end{theorem}

\begin{proof}[Proof of Theorem~\ref{Th:2KindCoord}]
Let $\widehat X'_1, \ldots, \widehat X'_N$ be
$\delta_\varepsilon$-homogeneous vector fields
from Theorem~\ref{Th:KarmVod}.
Consider mapping
$\widehat{\phi}' : \mathbb{R}^N \to \mathbb{R}^N$
defined as
\begin{multline*}
  \widehat{\phi}'(u_1, \ldots, u_N)
  = \exp (u_{L,1} \widehat X'_{L,1} + \ldots
    + u_{L,k_L} \widehat X'_{L,k_L})
  \circ \ldots \\
  \ldots \circ
  \exp (u_{2,1} \widehat X'_{2,1} + \ldots
    + u_{2,k_2} \widehat X'_{2,k_2})
  \circ
  \exp (u_{1,1} \widehat X'_{1,1} + \ldots
    + u_{1,k_1} \widehat X'_{1,k_1})(0).
\end{multline*}
Note that $\widehat{\phi}'$ is a $C^\infty$-diffeomorphism
of $\mathbb{R}^N$. Since $\widehat X'_j$ are homogeneous
with degrees~$\sigma_j$, for all $u, v \in \mathbb{R}^N$
we have
\[
  \delta_\varepsilon \circ \exp (
     u_1 \widehat X'_1 + \ldots + u_N \widehat X'_N
  )(v) =
  \exp (
    \varepsilon^{\sigma_1} u_1 \widehat X'_1 + \ldots
    + \varepsilon^{\sigma_N} u_N \widehat X'_N
  )
  (\delta_\varepsilon v),
\]
thus
$\delta_\varepsilon \circ \widehat{\phi}' =
\widehat{\phi}' \circ \delta_\varepsilon$.

By Theorem~\ref{Th:CurveDivergence} for every tuple
of constants $(u_1, \ldots, u_N)$ we have
\[
  \widehat d_\infty \Big(
    \exp \Big( \sum_{k=1}^N
    \varepsilon^{\sigma_k} u_k X_k \Big)(x),
    \theta_p \circ \exp \Big( \sum_{k=1}^N 
    \varepsilon^{\sigma_k} u_k
    \widehat X'_k \Big) \circ \theta_p^{-1} (x)
  \Big) = o(\varepsilon)
\]
where $o(\varepsilon)$ is uniform in $x$ belonging to
a neighborhood of $p$ and in $(u_1, \ldots, u_N)$ belonging
to a neighborhood of origin. Therefore
\[
  \widehat d_\infty \big(
    \phi_p(\delta_\varepsilon u),
    \theta_p \circ \widehat \phi'(\delta_\varepsilon u)
  \big) = o(\varepsilon)
\]
when $\varepsilon \to 0$ uniformly in $u$.
Quasimetric $\widehat{d}_\infty$ is homogeneous w.r.t.
$\theta_p$, thus
\begin{multline*}
  \frac{1}{\varepsilon} \widehat d_\infty \big(
    \phi_p(\delta_\varepsilon u),
    \theta_p \circ \widehat \phi'(\delta_\varepsilon u)
  \big)
  = \frac{1}{\varepsilon} \widehat d_\infty \big(
    \theta_p \circ \theta_p^{-1} \circ 
    \phi_p(\delta_\varepsilon u),
    \theta_p \circ \delta_\varepsilon \circ \widehat \phi'(u)
  \big) \\
  = \widehat d_\infty \big(
    \theta_p \circ \delta_\varepsilon^{-1} \circ
    \theta_p^{-1} \circ \phi_p(\delta_\varepsilon u),
    \theta_p \circ \widehat \phi'(u)
  \big) = o(1).
\end{multline*}
From that we conclude that there is a limit
\[
  \lim_{\varepsilon \to 0}
  \delta_\varepsilon^{-1} \circ \theta_p^{-1} \circ
  \phi_p \circ \delta_\varepsilon
  = \widehat{\phi}'
\]
uniform in a neighborhood of origin and the conditions
of Corollary~\ref{Corollary:CC1} are fulfilled.
Isometry of quasimetrics is given by map
$\mathcal{L} =
\theta_p \circ \widehat{\phi}' \circ \phi_p^{-1}$:
$\widetilde d_\infty(x, y) =
\widehat d_\infty(\mathcal{L} x, \mathcal{L} y)$.

In the case of $C^m$-smooth Carnot--Carath\'{e}odory space
both coordinate systems $\theta_p$ ans $\phi_p$ are
$C^m$-smooth as well, therefore from 
Theorem~\ref{Th:HomoNecessary} and 
Corollary~\ref{Corollary:CC3} it follows that conditions
of Corollary~\ref{Corollary:CC2} are fulfilled and the mapping 
$\mathcal{L}$ defines the isomorphism of Lie algebras:
$\widetilde X_k(x) =
\mathcal{L}^{-1}_* \widehat X_k(\mathcal{L} x)$,
$k = 1, \ldots, N$.
This ends the proof.
\end{proof}

{\sc
\noindent
Sergey Basalaev, \\
Sobolev Institute of Mathematics, \\
630090, Novosibirsk, Koptyug av., 4. \\
Novosibirsk State University, \\
630090, Novosibirsk, Pirogova st., 2. \\
E-mail: sbasalaev@gmail.com
}

\end{document}